\documentclass[a4paper, oneside, 11pt]{article}

\usepackage[margin=1in]{geometry}

\usepackage{listings}
\usepackage{xcolor}
\definecolor{codegreen}{rgb}{0,0.6,0}
\definecolor{codegray}{rgb}{0.5,0.5,0.5}
\definecolor{codepurple}{rgb}{0.58,0,0.82}
\definecolor{backcolour}{rgb}{0.94,0.96,0.96}

\lstdefinestyle{mystyle}{
  backgroundcolor=\color{backcolour}, commentstyle=\color{codegreen},
  keywordstyle=\color{magenta},
  numberstyle=\tiny\color{codegray},
  stringstyle=\color{codepurple},
  basicstyle=\ttfamily\normalsize,
  breakatwhitespace=false,         
  breaklines=true,                 
  captionpos=b,                    
  keepspaces=true,                 
  numbers=left,                    
  numbersep=5pt,                  
  showspaces=false,                
  showstringspaces=false,
  showtabs=false,                  
  tabsize=2
}

\usepackage{amsmath,amssymb,amsthm}
\theoremstyle{plain}							
\newtheorem{theorem}{Theorem}[section]
\newtheorem{proposition}[theorem]{Proposition}
\newtheorem{corollary}[theorem]{Corollary}		
												
\theoremstyle{definition}						
\newtheorem{example}[theorem]{Example}

\usepackage[usenames,dvipsnames]{pstricks}
\usepackage{pstricks-add}
\usepackage{epsfig}
\usepackage{pst-grad} 
\usepackage{pst-plot} 
\usepackage[space]{grffile} 
\usepackage{etoolbox} 
\makeatletter 
\patchcmd\Gread@eps{\@inputcheck#1 }{\@inputcheck"#1"\relax}{}{}
\makeatother
 
\usepackage{subcaption}

\begin{document}

\begin{center}
\Large\bfseries gcd-Pairs in $\mathbb{Z}_{n}$ and their graph representations

\bigskip\normalsize\normalfont
Wanchai Tapanyo\footnote{Corresponding author: wanchai.t@nsru.ac.th}, Tanyaton Tongpikul, Suphansa Kaewpradit

\bigskip\small
Division of Mathematics and Statistics, Nakhon Sawan Rajabhat University,\\ Nakhon Sawan, 60000, Thailand    
\end{center}

\normalsize

\begin{abstract}
    This research introduces a gcd-pair in $\mathbb{Z}_n$ which is an unordered pair $\{[a]_n, [b]_n\}$ of elements in $ \mathbb{Z}_n $ such that $0\leq a,b < n$ and the greatest common divisor $\gcd(a,b)$ divides $ n $. The properties of gcd-pairs in $ \mathbb{Z}_n $ and their graph representations are investigated. We also provide the counting formula of gcd-pairs in $ \mathbb{Z}_n $ and its subsets. The algorithms to find, count and check gcd-pairs in $ \mathbb{Z}_{n}$ are included.
\end{abstract}

\section{Introduction}

For a long time, Number Theory has been developed in many directions to create new branches of mathematics and excited mathematicians by advantages. Although new problems are so interesting, Number Theory still extends its own story by fascinating some people who love integers to create novel research.

Let $a, b\in\mathbb{Z}$. We say that $a$ is congruent to $b$ modulo $n$, denoted by $a \equiv b \mod n$, if $n$ divides $a-b$. It is an equivalence relation inducing a partition $\mathbb{Z}_n$ of $\mathbb{Z}$. The partition $\mathbb{Z}_n$ is a quotient ring of $\mathbb{Z}$ by $n\mathbb{Z}$. So we can characterize its elements into three types, i.e., the zero element, unit elements, and zero divisors. Throughout this research, We denote by $U(\mathbb{Z}_n)$ and $Z(\mathbb{Z}_n)$, the sets of unit elements and zero divisors, respectively.

Let us turn back to one of the easily understandable problems. For any nonnegative integer $n$, how many pairs of nonnegative integers $a$ and $b$ such that $a+b = n$?. Some authors \cite{TangNPartition,TangNPartition2,XiaoHuiNPartition} applied the number of solutions of the sum to study partitions of sets of nonnegative integers. The problem also extended to $\mathbb{Z}_n$, see more in \cite{CHENZmpartition,CHENZmpartition2,SunZnpartition} for further information. There are some studies investigating pairs of elements in $\mathbb{Z}_n$ constructed by other conditions. In 1999 \cite{ShanMutualMul}, Shan and Wang defined an unordered pair of elements in $\mathbb{Z}_n$, called a mutual multiple pair (MM pair), and investigated its properties. They also provided the counting formula of the number of all MM pairs in $\mathbb{Z}_n$. About 20 years later, Chalapathi et al. \cite{ChalapathiClassical} used the study of Shan and Wang as motivation to study another type of unordered pairs in $\mathbb{Z}_n$, named classical pairs. This type of pairing is determined by the least common multiple of remainders representing elements in $\mathbb{Z}_n$ considering whether the least common multiple is divisible by $n$. In this research, we define the new type of unordered pairs of residue classes in $\mathbb{Z}_n$ by considering whether $n$ is divisible by the least common divisor of remainders representing the classes. Therefore, the pairs will be straightforwardly called \emph{gcd-pairs}. The properties and enumerating formulas of gcd-pairs are investigated. The algorithms for seeking, counting, and checking gcd-pairs are included.

\section{Some Properties of gcd-Pairs}

In this section, we introduce a gcd-pair in $ \mathbb{Z}_n $ and provide its properties. Let $ [a]_n, [b]_n \in \mathbb{Z}_n $ be represented by nonnegative integers $ a $ and $ b $ less than $ n $. An unordered pair $\{[a]_n, [b]_n\}$ is called a \emph{gcd-pair} in $\mathbb{Z}_n$ if and only if $\gcd(a,b)|n$ and the set of all gcd-pairs in $\mathbb{Z}_n$ is denoted by $\nu_n$, that is,
    \[ \nu_n = \big\{\{[a]_n, [b]_n\} : 0\leq a,b < n \text{ and } \gcd(a,b)\mid n \big\} \]
If $A\subseteq \mathbb{Z}_n$ we denote by $\nu_{n,A}$ the set of all gcd-pairs in $\mathbb{Z}_n$ restricted to $A$, i.e.,
    \[ \nu_{n,A} = \nu_n \cap \big\{\{[a]_n, [b]_n\} : [a]_n, [b]_n\in A \big\} \]
Since the gcd-pair is determined by the greatest common divisor of the nonnegative integers less than $n$ representing those two congruence classes, we simplify the notations of all elements in $\mathbb{Z}_n$ by writing and considering them as nonegative numbers less than $n$. Next, we provide the first example listing all gcd-pairs in $\mathbb{Z}_6$.

\begin{example}\label{ExListof6}
Consider $\mathbb{Z}_6 = \{0,1,2,3,4,5\}$. The greatest common divisors of all pairs of elements in $\mathbb{Z}_6$ are listed below.
\[	
\begin{array}{ccccccc}
    \gcd(0,1)=1 &&\gcd(0,2)=2 &&\gcd(0,3)=3 &&\gcd(0,4)=4\\
    \gcd(0,5)=5 &&\gcd(1,1)=1 &&\gcd(1,2)=1 &&\gcd(1,3)=1\\
    \gcd(1,4)=1 &&\gcd(1,5)=1 &&\gcd(2,2)=2 &&\gcd(2,3)=1\\
    \gcd(2,4)=2 &&\gcd(2,5)=1 &&\gcd(3,3)=3 &&\gcd(3,4)=1\\
    \gcd(3,5)=1 &&\gcd(4,4)=4 &&\gcd(4,5)=1 &&\gcd(5,5)=5\\
\end{array}
\]
Then $\nu_n$ consists of unordered pairs, namely, $\{0,1\},$ $\{0,2\},$ $\{0,3\},$ $\{1,1\},$ $\{1,2\},$ $\{1,3\},$ $\{1,4\},$ $\{1,5\},$ $\{2,2\},$ $\{2,3\},$ $\{2,4\},$ $\{2,5\},$ $\{3,3\},$ $\{3,4\},$ $\{3,5\},$ $\{4,5\}$.
\end{example}

Next, we provide some results characterizing gcd-pairs in $\mathbb{Z}_n$. Let us consider the condition $\gcd(a,b)\mid n$. It is true only for $a$ and $b$ that their all common divisors are factors of $n$. Therefore, one of the simple conditions which makes $\{a,b\}$ be a gcd-pair is that $a$ or $b$ are factors of $n$. The proof is done by the fact that $\gcd(a,b)$ divides $a$ and $b$. 

\begin{theorem}\label{ThmgcdpairPro}
For every $a,b\in \mathbb{Z}_n$, if $a\mid n$, then $\{a,b\}\in \nu_n$.
\end{theorem}

\begin{example}
Consider $ \mathbb{Z}_{6} = {\{0,1,2,3,4,5}\} $. Since $ 1,2 $ and $3 $ are all factors of $ 6 $ in $\mathbb{Z}_6$, unordered pairs consisting at least one of these numbers are gcd-pairs in $\mathbb{Z}_6$, i.e.,
	$\{0,1\},$ $\{0,2\},$ $\{0,3\},$ $\{1,1\},$ $\{1,2\},$ $\{1,3\},$ $\{1,4\},$ $\{1,5\},$ $\{2,2\},$ $\{2,3\},$ $\{2,4\},$ $\{2,5\},$ $\{3,3\},$ $\{3,4\},$ $\{3,5\}$.
\end{example}

We know that $1$ divides all integers. Thus, every pair of relatively prime numbers $a$ and $b$, $\{a,b\}$ is a gcd-pair. One can show that the reverse implication of this statement is not true in general. Hence, we need some exact conditions added in the statement, namely, $a$ or $b$ are relatively prime with $n$, or equivalently, $a$ or $b$ are in $U(\mathbb{Z}_n)$.

\begin{example}\label{ExListof9}
In $\mathbb{Z}_9$, all pairs of numbers listed below are gcd-pairs in $\mathbb{Z}_9$,

$\{0,3\},$ $\{3,3\},$ $\{3,6\},$ $\{0,1\},$ $\{1,1\},$ $\{1,2\},$ $\{1,3\},$ $\{2,3\},$ $\{1,4\},$ $\{3,4\},$ 

$\{1,5\},$ $\{2,5\},$ $\{3,5\},$ $\{4,5\},$ $\{1,6\},$ $\{5,6\},$ $\{1,7\},$ $\{2,7\},$ $\{3,7\},$ $\{4,7\},$ 

$\{5,7\},$ $\{6,7\},$ $\{1,8\},$ $\{3,8\},$ $\{5,8\},$ $\{7,8\}$. \\
The first 3 pairs provide the same greatest common divisor equal to $3$. For the others, the numbers are relatively prime in every pair.
\end{example}

\begin{theorem}
For every $a,b\in \mathbb{Z}_n$ with $\gcd(a,n)=1$, if $\{a,b\}\in \nu_n$, then $\gcd(a,b)= 1$.
\end{theorem}
\begin{proof}
By contrapositive, we suppose that $\gcd(a,b)\neq 1$. Since $\gcd(a,n)=1$ and $\gcd(a,b)\mid b$, this implies that $\gcd(a,b)\mid n$. Thus, $\{a,b\}\notin \nu_n$. 
\end{proof}

The contrapositive of the previous theorem is also useful if we know all elements in $U(\mathbb{Z}_n)$. We can check that $\{a,b\}$ is not a gcd-pair by cosidering that $\gcd(a,b)\neq 1$.

\begin{example}
Consider $\mathbb{Z}_9 = \{0,1,2,\ldots, 8\}$ and $U(\mathbb{Z}_9)=\{0,1,2,4,5,7,8\}$. Since $\gcd(4,6)=2\neq 1$ and $\gcd(4,8) =4\neq 1$, we have $\{4,6\}, \{4,8\}\notin \nu_n$.
\end{example}

\section{gcd-Pairs Enumeration}

This section provides the formulas to count gcd-pair. Some relations between the numbers of gcd-pairs in $\mathbb{Z}_n$ restricted to its subsets are also considered. The first theorem of this part provides a formula related to the Euler's totient function $\phi$ for counting elements in $\nu_{p^k}$, where $p$ is a prime number and $k$ is a positive integer.

\begin{theorem}\label{ThmNuP}
Let $p$ be a prime number and $k$ be a positive integer. Then 
    $$|\nu_{p^k}| = k+\sum\limits_{i=1}^k\sum\limits_{j=1}^{p^i-1}\phi(j).$$
In particular, $|\nu_p| = 1+\sum\limits_{k=1}^{p-1}\phi(k)$.
\end{theorem}
\begin{proof}
    We prove this theorem by mathematical induction. Let us consider the basis step $k=1$. We will show that $|\nu_{p}| = 1+\sum\limits_{j=1}^{p-1}\phi(j).$ Since $p$ is prime, only positive integers $1$ and $p$ divide $p$. Then for every $a, b\in \mathbb{Z}_p$, $\{a,b\}\in \nu_n$ if and only if $\gcd(a,b) = 1$. For each $b = 1, 2,\ldots, p-1$, we see that the number of gcd-pairs $\{a,b\}$ in $\mathbb{Z}_p$ such that $0< a\leq b$ is equal to $\phi(b)$. Combine with $\{0,1\}$, so we have all gcd-pairs in $\mathbb{Z}_n$. Therefore, $|\nu_{p}| = 1+\sum\limits_{j=1}^{p-1}\phi(j).$
    
    For the inductive step, we assume that $|\nu_{p^k}| = k+\sum\limits_{i=1}^k\sum\limits_{j=1}^{p^i-1}\phi(j).$ Then consider a gcd-pair contained in $\nu_{p^{k+1}}$. Its greatest common divisor is one of integers $1, p, p^2,\ldots, p^{k}$. The number of gcd-pairs, whose greatest common divisors are $1$, is equal to $1+\sum\limits_{j=1}^{p^{k+1}-1}\phi(j).$ The rests are gcd-pairs whose greatest common divisor are greater than $1$. They consist of numbers in the set $\{0, p, 2p,\ldots, (p^{k+1}-1)p\}$ which are divisible by $p$. We see that for every $j= 1,2,\ldots,k$ and every pair of nonnegative integers $r,s \leq p^{k+1}-1$, it is obvious that $\gcd(rp,sp)= p^j$ if and only if $\gcd(r,s)= p^{j-1}$. This implies that the number of gcd-pairs consisting of numbers in $\{0, p, 2p,\ldots, (p^{k+1}-1)p\}$ is equal to $|\nu_{p^k}|$. Thus,
    \begin{align*}
        |\nu_{p^{k+1}}| &= 1+\sum\limits_{j=1}^{p^{k+1}-1}\phi(j) + |\nu_{p^k}|\\
            &= 1+\sum\limits_{j=1}^{p^{k+1}-1}\phi(j) + k+\sum\limits_{i=1}^k\sum\limits_{j=1}^{p^i-1}\phi(j)\\
            &= (k+1) + \sum\limits_{i=1}^{k+1}\sum\limits_{j=1}^{p^i-1}\phi(j).
    \end{align*}
    Therefore, the proof is now completed.
\end{proof}

In the case of a composite number $n$, if a prime $p$ divides $n$, $\{p, kp\}$ is a gcd-pair in $\mathbb{Z}_{n}$ for every $k=0,1,\ldots p-1$. We see that gcd-pairs in this formula provide the greatest common divisor $p$. Therefore, we obtain the following corollary immediately.

\begin{corollary}
For every composite number $n$, $|\nu_n| > 1+\sum\limits_{k=1}^{n-1}\phi(k)$
\end{corollary}

\begin{example}
    From Example \ref{ExListof9} and Example \ref{ExListof6}, we have $|\nu_9| = 26$ and $|\nu_6| = 16$, respectively. Then
    \begin{flalign*}
        &&|\nu_9| &= 2 + (1+1)+(1 + 1 + 2 + 2 + 4 + 2 + 6 + 4)&&\\
        &&  &= 2 + \sum\limits_{j=1}^{2}\phi(j) + \sum\limits_{k=1}^{8}\phi(k)&&\\
        &&  &= 2 + \sum\limits_{i=1}^2\sum\limits_{j=1}^{3^i-1}\phi(j),&&\\
        &\text{and}& |\nu_6| &> 1 + 1 + 1 + 2 + 2 + 4 = 1 +\sum\limits_{k=1}^{5}\phi(k).&&
    \end{flalign*}
\end{example}

Let $D_n$ be the set of all nontrivial positive divisors of $n$, that is, $D_n = \{d\in \mathbb{N} : d\mid n \text{ and } d\neq 1\}$. For every $d\in D_n$ define $S'_d = \{rd\in \mathbb{Z}_n : \gcd(rd, n) = d \text{ and } 1\leq r < \frac{n}{d}\}$. From \cite[Theorem 8]{SajanaNoThmProp}, $\{S'_d : d\in D_n\}$ is a partition of $Z(\mathbb{Z}_n)$, i.e., $Z(\mathbb{Z}_n)= \bigsqcup\limits_{d\in D_n} S'_d$. If $n$ is composite, the example below shows that there is at least one gcd-pair consisting of elements from different cells in that partition.

\begin{example}
    In $\mathbb{Z}_6$, we have $S'_2= \{2, 4\}$ and $S'_3 = \{3\}$. It is easily seen that $\{2, 3\}$ and $\{3, 4\}$ are gcd-pairs in $\mathbb{Z}_6$ such that numbers in each pair belong to different cells in the partition $\{S'_2, S'_3\}$ of $Z(\mathbb{Z}_6)$ 
\end{example}

\begin{theorem}\label{ThmNuZZn}
$|\nu_{n,Z(\mathbb{Z}_n)}| \geq \sum\limits_{d\in D_n}|\nu_{m,U(\mathbb{Z}_m)}| $ where $m = \frac{n}{d}$.
\end{theorem}
\begin{proof}
Let $d\in D_n$ and $m=\frac{n}{d}$, so $\gcd(rd, n) = d$ if and only if $\gcd(r,m) = \gcd(r,\frac{n}{d}) = 1$. Thus, $rd\in S'_d$ if and only if $r\in U(\mathbb{Z}_m)$. Now, suppose that $rd, sd\in S'_d$, or equivalently, $r, s \in U(\mathbb{Z}_m)$. We have $\gcd(rd, sd)\mid n$ if and only if $\gcd(r,s)\mid m$. This implies that $\{rd, sd\}\in \nu_{n,S'_d}$ if and only if $\{r,s\}\in \nu_{m,U(\mathbb{Z}_m)}$. This matching provides one-to-one correspondence between $\nu_{n,S'_d}$ and $\nu_{m,U(\mathbb{Z}_m)}$. By the previous example, $|\nu_{n,Z(\mathbb{Z}_n)}| \geq \sum\limits_{d\in D_n}|\nu_{m,U(\mathbb{Z}_m)}| $ where $m = \frac{n}{d}$.
\end{proof}

\begin{corollary}
    Let $p$ and $q$ be different prime numbers, then $$|\nu_{pq,Z(\mathbb{Z}_{pq})}|\geq |\nu_p| + |\nu_q| + p + q - 5.$$
\end{corollary}
\begin{proof}
    By the definition of $S'_d$, one can compute that 
    $$S'_p= \{p,2p,\ldots,(q-1)p\} \text{ and } S'_q= \{q,2q,\ldots,(p-1)q\}.$$
    Since $D_{pq} = \{p, q\}$, $Z(\mathbb{Z}_{pq}) = S'_p \sqcup S'_q$. By Theorem \ref{ThmNuZZn}, $|\nu_{pq,Z(\mathbb{Z}_{pq})}|\geq |\nu_{p,U(\mathbb{Z}_p)}| + |\nu_{q,U(\mathbb{Z}_q)}|$. Since $p$ is prime, $U(\mathbb{Z}_p)= \mathbb{Z}_p\setminus\{0\}$. We can see that $\{0,1\}$ is the only one gcd-pair in $\mathbb{Z}_p$ containing $0$. Then $|\nu_{p,U(\mathbb{Z}_p)}| = |\nu_p| - 1 $. By the same arguments, we have $|\nu_{q,U(\mathbb{Z}_q)}| = |\nu_q| - 1 $. Hence,
    \begin{equation}\label{EqPQ1stmodify}
        |\nu_{pq,Z(\mathbb{Z}_{pq})}|\geq |\nu_p| + |\nu_q| - 2.
    \end{equation}
    Next, consider unordered pairs of the forms $\{p,x\}$ and $\{y,q\}$ where $x\in S'_q$ and $y\in S'_p$. We see that $\gcd(p,x)= \gcd(y,q) = 1$. Then they are gcd-pairs in $\mathbb{Z}_{pq}$ for every $x\in S'_q$ and $y\in S'_p$. The number of gcd-pairs of these forms is equal to $|S'_q| + |S'_p| -1 = p+q-3$. This count focuses gcd-pairs whose elements come from different cells in the partition that never occurs between $\mathbb{Z}_p$ and $\mathbb{Z}_q$. Thus, the number $p+q-3$ can be added to the right-hand side of \eqref{EqPQ1stmodify} without overcounting. Therefore, 
    \[|\nu_{pq,Z(\mathbb{Z}_{pq})}|\geq |\nu_p| + |\nu_q| - 2 + p+q-3 = |\nu_p| + |\nu_q| + p+q-5. \qedhere\]
\end{proof}

\begin{example}
    Consider $Z(\mathbb{Z}_{15}) = \{3,5,6,9,10,12\}$. Then
    	\begin{align*}
		\nu_{15,Z(\mathbb{Z}_{15})} = \big\{&\{3,3\}, \{3,5\}, \{3,6\}, \{3,9\}, \{3,10\}, \{3,12\}, \{5,5\}, \{5,6\},\\
		& \{5,9\}, \{5,10\}, \{5,12\}, \{6,9\}, \{9,10\}, \{9,12\}\big\}.
	\end{align*}
	From Theorem \ref{ThmNuP}, we have $|\nu_3| = 3$ and $|\nu_5| = 7$. If we let $p = 3$ and $q = 5$, then
	$|\nu_{15,Z(\mathbb{Z}_{15})}| = 14 \geq 3 + 7 + 3 + 5 - 5 = |\nu_3| + |\nu_5| + p + q - 5.$
\end{example}

\begin{corollary}
    Let $p$ be a prime number, then
    \begin{enumerate}
        \item $|\nu_{2p,Z(\mathbb{Z}_{2p})}| = |\nu_p| + p - 1$ for every $p\neq 2$,
        \item $|\nu_{3p,Z(\mathbb{Z}_{3p})}| = |\nu_p| + p + \left\lceil\frac{p-1}{2}\right\rceil$ for every $p\neq 3$.
    \end{enumerate}
\end{corollary}
\begin{proof}
    The methodology is similar to the two results above. We begin with the first item. In this case, we have $Z(\mathbb{Z}_{2p}) = S'_2 \sqcup S'_p $ such that $$S'_2 = \{2, 2(2), \ldots, 2(p-1)\} \text{ and } S'_p = \{p\}.$$ Then gcd-pairs in $\nu_{2p,Z(\mathbb{Z}_{2p})}$ will be separated into two groups for easier counting as follows:
    \begin{description}
        \item[Group 1] Gcd-pairs consist of only elements in $S'_2$.
        \item[Group 2] At least one of elements in each gcd-pair is an element in $S'_p$.
    \end{description}
    We can compute that Group 1 contains $|\nu_{p,U(\mathbb{Z}_p)}|$ gcd-pairs and Group 2 contains $|S'_2| + 1$ gcd-pairs. Therefore,
    \[|\nu_{2p,Z(\mathbb{Z}_{2p})}| = |\nu_{p,U(\mathbb{Z}_p)}| + |S'_2| + 1 = (|\nu_p| - 1) + (p - 1) + 1 = |\nu_p| + p - 1.\]
    
    For the second item, we have $Z(\mathbb{Z}_{3p}) = S'_3 \sqcup S'_p $ such that 
        $$S'_3 = \{3, 3(2), \ldots, 3(p-1)\} \text{ and } S'_p = \{p, 2p\}.$$ 
    The proof follows from that of the first item. The number of gcd-pairs of these two cases are different only in Group 2. In the case of $Z(\mathbb{Z}_{3p})$, the number of gcd-pairs in Group 2 is $|S'_3| + \left\lceil\frac{|S'_3|}{2}\right\rceil + 2$. Therefore,
    \begin{align*}
        |\nu_{3p,Z(\mathbb{Z}_{3p})}| 
        &= |\nu_{p,U(\mathbb{Z}_p)}| + |S'_3| + \left\lceil\frac{|S'_3|}{2}\right\rceil + 2\\
        &= (|\nu_p| - 1) + (p - 1) + \left\lceil\frac{p-1}{2}\right\rceil + 2\\
        &= |\nu_p| + p + \left\lceil\frac{p-1}{2}\right\rceil.\qedhere
    \end{align*}
\end{proof}

\begin{example}
    Consider $Z(\mathbb{Z}_{6}) = \{2,3,4\}$. Then 
    	$$ \nu_{6,Z(\mathbb{Z}_{6})} = \big\{\{2,2\}, \{2,3\}, \{2,4\}, \{3,3\}, \{3,4\}\big\} $$
    If we let $6=2p$, we have $p = 3$. From Theorem \ref{ThmNuP}, $|\nu_p|=|\nu_3| = 3$. Thus,
    \[\nu_{6,Z(\mathbb{Z}_{6})} = 5 = 3+3-1 = |\nu_p|+p-1.\]
    If we let $6=3p$, we have $p = 2$. From Theorem \ref{ThmNuP}, $|\nu_p|=|\nu_2| = 2$. Thus,
    \[\nu_{6,Z(\mathbb{Z}_{6})} = 5 = 2+2+1 = |\nu_p|+p+.\left\lceil\frac{p-1}{2}\right\rceil.\]
\end{example}

\begin{corollary}
    Let $p$ be a prime number. Then $$|\nu_{p^k,Z(\mathbb{Z}_{p^k})}| = |\nu_{p^{k-1}}| - k +1,$$
    for every integer $k\geq 2$. In particular, $|\nu_{p,Z(\mathbb{Z}_{p})}| = 0$.
\end{corollary}
\begin{proof}
    In case $k=1$, $|\nu_{p,Z(\mathbb{Z}_{p})}| = 0$ because of $Z(\mathbb{Z}_{p})=\emptyset$. Now we assume that $k \geq 2$. In this case, $Z(\mathbb{Z}_{p^k}) = \{p,2p,\ldots ,({p^{k-1}}-1)p\}$. Then for every pair of positive integers $r, s\leq p^{k-1}-1$, we obtain that $\gcd(rp,sp)\mid p^k$ if and only if $\gcd(r,s)\mid p^{k-1}$. This provide one-to-one matching between all elements $\nu_{p^k,Z(\mathbb{Z}_{p^k})}$ and $\nu_{p^{k-1},(\mathbb{Z}_{p^{k-1}}\setminus\{0\})}$, so $|\nu_{p^k,Z(\mathbb{Z}_{p^k})}|=|\nu_{p^{k-1},(\mathbb{Z}_{p^{k-1}}\setminus\{0\})}|$. Since $\{0,p^j\}$ is also a gcd-pair for every $j=1,2,\ldots,k-2$, we have
    \[|\nu_{p^k,Z(\mathbb{Z}_{p^k})}| = |\nu_{p^{k-1}}| - (k-1)= |\nu_{p^{k-1}}| - k + 1.\qedhere\]
\end{proof}

\begin{example}
    Let $p=2$ and $k=3$, so $p^k = 8$ and $p^{k-1} = 4$. Since $Z(\mathbb{Z}_8)= \{2,3,6\}$ and $\mathbb{Z}_4 = \{0,1,2,3\}$, we have
    \begin{flalign*}
        &&\nu_{8,Z(\mathbb{Z}_8)} &= \big\{\{2,2\}, \{2,4\}, \{2,6\}, \{4,4\}, \{4,6\}\big\}&&\\
        &\text{and}& \nu_4 &= \big\{\{0,1\}, \{0,2\}, \{1,1\}, \{1,2\}, \{1,3\}, \{2,2\}, \{2,3\}\big\}.&&
    \end{flalign*}
    Hence, $|\nu_{8,Z(\mathbb{Z}_8)}| = 7-3+1 = |\nu_4| - k + 1$.
\end{example}

\section*{Coding}

Next, we provides algorithms to find, count, and check gcd-pairs in $\mathbb{Z}_n$ written in Python. Notice that the first two codes are for finding and counting gcd-pairs.

\begin{lstlisting}[language=Python, caption= ]
n = int(input("enter the number n: "))
gcdpair_no = 0 # gsd-pairs counter
for a in range(0,n):
    for b in range(a,n):
        aa = a
        bb = b if b != 0 else n+1
        while aa != 0: # calculate gcd(a,b)
            d = aa
            aa = bb%d
            bb = d
        if n%bb == 0: # gcd-pair checking (definition)
            print("{%d,%d}" % (a,b))
            gcdpair_no += 1
print("The number of gcd-pairs is", gcdpair_no)
\end{lstlisting}

Line 4 to 8 of the second code written using Theorem \ref{ThmgcdpairPro} replacing the code written from the definition of a gcd-pair. Thus, we can skip the sub-algorithm which finds the greatest common divisor of two numbers, so the performance of the second algorithm will be better. However, if we consider the set of nonnegative integers less than $n$, we will see that the factors of $n$ are rare. Thus, the performances of these two codes are not significantly different for every small number $n$ which have not many factors.

\begin{lstlisting}[language=Python, caption= ]
n = int(input("enter the number n: "))
gcdpair_no = 0 # gsd-pairs counter
for a in range(1,n):
    if n%a == 0: # gcd-pair checking (theorem) 
        print("{%d,%d}" % (a,0))
        for b in range(a,n):
            print("{%d,%d}" % (a,b))
        gcdpair_no = gcdpair_no + (n-a) + 1
    else:
        for b in range(a,n):
            aa = a
            bb = b
            while aa != 0: # calculate gcd(a,b)
                d = aa
                aa = bb%d
                bb = d
            if n%bb == 0: # gcd-pair checking (definition)
                print("{%d,%d}" % (a,b))
                gcdpair_no += 1
print("The number of gcd-pairs is", gcdpair_no)
\end{lstlisting}

We end this section with a code for gcd-pair checking. The input numbers $a$ and $b$ can be negative or greater than $n$ since we consider them as representatives of congruence classes. The algorithm will divide them by $n$ to find and use their remainders to investigate whether $\{a,b\}$ is a gcd-pair in $\mathbb{Z}_n$.

\begin{lstlisting}[language=Python, caption= ]
n = int(input("enter the number n: "))
a = int(input("enter the number a: "))
b = int(input("enter the number b: "))
ra = a%n # remainder of a
rb = b%n # remainder of b
if ra == 0 and rb == 0: # avoid gcd(0,0)
    print("{%d,%d} is not a gcd-pair in Z%d" % (a,b,n))
else:
    aa = ra
    bb = rb
    while aa != 0: # calculate gcd(ra,rb)
        d = aa
        aa = bb%d
        bb = d
    if n%bb == 0: # gcd-pair checking (Definition)
        print("{%d,%d} is a gcd-pair in Z%d" % (a,b,n))
    else:
        print("{%d,%d} is not a gcd-pair in Z%d" % (a,b,n))
\end{lstlisting}

\section{Graphs Representing gcd-Pairs}

The set $\nu_n$ can be considered as a family of edges of a graph with the vertex set $\mathbb{Z}_n$. It is interesting what properties of the graph has, so we investigate some those properties in this section. For the first illustration we demonstrate by the graph representing gcd-pairs in $\mathbb{Z}_5$

\begin{figure}[!h]
\centering
\psscalebox{1.0 1.0} 
{
\begin{pspicture}(0,-4.3999996)(3.31,-1.7634615)
\psdots[linecolor=black, dotsize=0.24](1.6640003,-3.0817308)
\psdots[linecolor=black, dotsize=0.24](0.46400025,-1.8817308)
\psdots[linecolor=black, dotsize=0.24](2.8640003,-1.8817308)
\psdots[linecolor=black, dotsize=0.24](2.8640003,-4.2817307)
\psdots[linecolor=black, dotsize=0.24](0.46400025,-4.2817307)
\psline[linecolor=black, linewidth=0.04](0.46400025,-4.2817307)(2.8640003,-1.8817308)
\psline[linecolor=black, linewidth=0.04](0.46400025,-1.8817308)(2.8640003,-4.2817307)
\psline[linecolor=black, linewidth=0.04](2.8640003,-1.8817308)(2.8640003,-4.2817307)(0.46400025,-4.2817307)
\psarc[linecolor=black, linewidth=0.04, dimen=outer](1.3640002,-2.781731){0.4}{325.14334}{311.9872}
\rput[bl](1.5680002,-3.6257308){$1$}
\rput[bl](0.0,-2.0737307){$0$}
\rput[br](3.3000002,-2.0737307){$2$}
\rput[tl](0.0,-4.073731){$4$}
\rput[t](3.2000003,-4.073731){$3$}
\end{pspicture}
}
\caption{The graph representing gcd-pairs in $\mathbb{Z}_5$}
\end{figure}
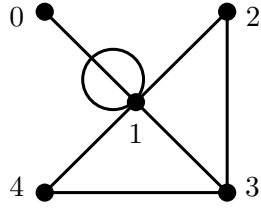

Let $G_n = (\mathbb{Z}_n, \nu_n)$ be a graph. We can see that the graph $G_n$ is more complicated for the larger number $n$. Some trivial properties about isomorphism, domination numbers, and connectivity of the graph $G_n$ are provided in the proposition below.

\begin{proposition}
Let $G_n = (\mathbb{Z}_n, \nu_n)$ be a graph. Then;
\begin{enumerate}
    \item $G_m$ is an isomorphic subgraph of $G_n$ whenever $m$ divides $n$.
    \item $G_n$ contains a maximal star subgraph of order $n$.
    \item The domination number $\gamma(G_n) = 1$.
    \item $G_n$ is connected.
\end{enumerate}
\end{proposition}
\begin{proof}
    For the first item, if $\gcd(a,b) $ divides $ m$, then $\gcd(a,b) $ divides $ n$. Thus the edge $\{a,b\}$ in $G_m$ is also an edge in $G_n$. This implies that $G_m$ can be embedded in $G_n$. Next, we see that $\{1,a\}$ is an edge in $G_n$ for every $a\in \mathbb{Z}_n$. Therefore, a star graph with the internal vertex $1$ together with $n-1$ leaves is a subgraph of $G_n$. The last two items are consequences of the previous one.
\end{proof}

\begin{figure}[!h]
    \centering
    \begin{subfigure}[b]{0.45\textwidth}
    \centering
\psscalebox{1.0 1.0} 
{
\begin{pspicture}(0,-4.7508683)(3.4422693,-3.812593)
\rput[bl](1.622269,-4.474862){$1$}
\psdots[linecolor=black, dotsize=0.24](1.718269,-3.9308622)
\psdots[linecolor=black, dotsize=0.24](0.11826904,-3.9308622)
\psdots[linecolor=black, dotsize=0.24](3.318269,-3.9308622)
\rput[bl](0.022269042,-4.474862){$0$}
\rput[bl](3.222269,-4.474862){$2$}
\psline[linecolor=black, linewidth=0.04](0.11826904,-3.9308622)(3.318269,-3.9308622)
\psarc[linecolor=black, linewidth=0.04, dimen=outer](1.718269,-4.330862){0.4}{87.87891}{87.39744}
\end{pspicture}
}\vspace{1.5cm}
    \caption{$G_3$.}
    \end{subfigure}
    ~ 
    \begin{subfigure}[b]{0.45\textwidth}
    \centering
\psscalebox{1.0 1.0} 
{
\begin{pspicture}(0,-4.3999567)(4.84034,0.44046494)
\psarc[linecolor=black, linewidth=0.04, dimen=outer](3.2601764,-1.9397893){0.4}{47.16669}{37.07866}
\rput[bl](2.3241763,-3.3237894){$1$}
\rput[tl](0.3241764,-2.9237893){$0$}
\rput[t](4.3241763,-2.9237893){$2$}
\psline[linecolor=red, linewidth=0.04, linestyle=dashed, dash=0.17638889cm 0.10583334cm](0.82017636,-2.7797892)(4.0201764,-2.7797892)
\psline[linecolor=black, linewidth=0.04](2.4201765,-2.7797892)(2.4201765,-1.1797893)
\psline[linecolor=black, linewidth=0.04](1.2201763,-1.5797893)(2.4201765,-2.7797892)(3.6201763,-1.5797893)
\psdots[linecolor=black, dotsize=0.24](1.2801764,-1.6397892)
\psdots[linecolor=black, dotsize=0.24](2.4201765,-1.1797893)
\psdots[linecolor=black, dotsize=0.24](3.5601764,-1.6397892)
\psarc[linecolor=black, linewidth=0.04, dimen=outer](2.4201765,-2.7797892){1.6}{179.40318}{134.46336}
\rput[b](3.9561763,-1.7717893){$3$}
\rput[bl](2.3241763,-0.92378926){$4$}
\rput[bl](0.7561764,-1.7717893){$5$}
\psarc[linecolor=black, linewidth=0.04, dimen=outer](3.2201765,-1.9797893){1.2}{-41.992718}{-224.12488}
\psarc[linecolor=black, linewidth=0.04, dimen=outer](3.0201764,-1.3797892){1.8}{-54.17327}{-172.55327}
\psdots[linecolor=black, dotsize=0.24](0.82017636,-2.7797892)
\psdots[linecolor=black, dotsize=0.24](4.0201764,-2.7797892)
\psarc[linecolor=black, linewidth=0.04, dimen=outer](3.6201763,-2.7797892){0.4}{0.0}{357.70938}
\psarc[linecolor=red, linewidth=0.04, linestyle=dashed, dash=0.17638889cm 0.10583334cm, dimen=outer](2.4201765,-3.1797893){0.4}{90.0}{82.04599}
\psdots[linecolor=black, dotsize=0.24](2.4201765,-2.7797892)
\psarc[linecolor=black, linewidth=0.04, dimen=outer](1.8201764,-1.3797892){1.8}{-6.5632415}{-125.00865}
\psarc[linecolor=black, linewidth=0.04, dimen=outer](2.4201765,-1.5797893){1.2}{-2.5447211}{-178.99937}
\end{pspicture}
}
    \caption{$G_6$.}
    \end{subfigure}
    \caption{The graph $G_3$ embedded in the graph $G_6$.}
\end{figure}
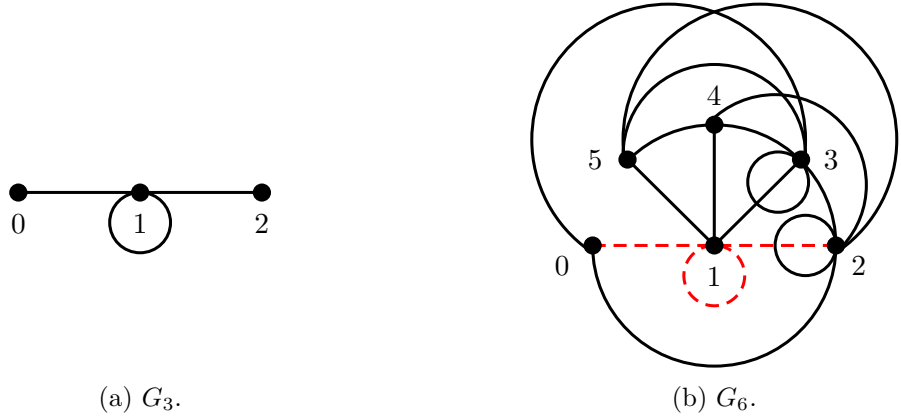

Next, we study paths and cycles in the graph $G_n$. A \emph{path} is a sequence of edges joining a sequence of distinct vertices of a graphs. In this paper we represent a path by $(v_1, v_2,\ldots, v_m)$, an $m$-tuple of different vertices. It is obvious that the edge joining $v_i$ and $v_{i+1}$ is $\{v_i, v_{i+1}\}$. A \emph{cycle} is a path that $v_1$ and $v_m$ are identical. A cycle with $3$ vertices will be called a \emph{triangle}. For any graph $G$, a path (resp. cycle) which visits all vertices of $G$ is called a \emph{Hamiltonian path (resp. cycle)}. A graph containing a Hamiltonian path (resp. cycle) is called a \emph{traceable (resp. Hamiltonian)} graph. A path or a cycle in a graph $G$ is said to be \emph{maximal} if the path or the cycle reach the possible maximum numbers of vertices.

\begin{proposition}
Let $G_n = (\mathbb{Z}_n, \nu_n)$ be a graph. Then;
\begin{enumerate}
    \item $G_n$ contains triangles if and only if $n\geq 4$.
    \item $G_n$ is a traceable graph.
    \item $G_n$ is a Hamiltonian graph if $n$ is even number greater than $2$. Otherwise, $G_n$ contains a maximal cycle of order $n-1$. 
\end{enumerate}
\end{proposition}
\begin{proof}
    We can draw the graphs $G_2$ and $G_3$, so we can see that there is no any cycle contained in the graphs. Now we suppose that $n\geq 4$. We see that $(1, a, a+1, 1)$ is a triangle in $G_n$ for every $a= 2,3,\ldots,n-2$. This show that $G_n$ contains triangles precisely for $n\geq 4$. Moreover, we know that $\gcd(a,a+1) = 1$ for every $a= 0, 1, \ldots, n-2$. Thus $(0, 1, 2, \ldots, n-1)$ is a Hamiltonian path in $G_n$, so the graph $G_n$ is a traceable graph.
    
    For the last one, we assume that $n$ is an even number greater than $2$. Then $(0, 2, 3,\ldots, n-1, 1, 0)$ is a Hamiltonian cycle in $G_n$. Thus $G_n$ is Hamiltonian. Next, we let $n$ be odd. Then all even vertices are not adjacent to each others. Suppose that $(0, h_1, h_2,\ldots, h_{n-1}, 0)$ is a Hamiltonian cycle in $G_n$. Then vertices in the cycle are odd and even alternately. If we consider the parities of vertices of the cycle from left to right, we will obtain that $h_1$ is odd and $h_{n-1}$ will be even. This is impossible since $h_{n-1}$ is adjacent to $0$ which is also even. Therefore, $G_n$ does not contain a Hamiltonian cycle. However, $(1, 2, \ldots, n-1, 1)$ is a cycle of order $n-1$ in $G_n$. Hence $G_n$ contains a maximal cycle of order $n-1$.
\end{proof}

\begin{figure}[!h]
    \centering
    \begin{subfigure}[b]{0.45\textwidth}
    \centering
\psscalebox{1.0 1.0} 
{
\begin{pspicture}(0,-4.796927)(5.3162513,-0.7627806)
\psarc[linecolor=red, linewidth=0.04, dimen=outer](2.0959997,-3.182927){1.6}{318.0128}{136.2364}
\psdots[linecolor=black, dotsize=0.2](3.6959999,-3.182927)
\psdots[linecolor=black, dotsize=0.24](0.95599973,-2.0429268)
\psdots[linecolor=black, dotsize=0.24](3.2359998,-2.0429268)
\psdots[linecolor=black, dotsize=0.24](2.0959997,-1.5829268)
\rput[t](3.9999998,-3.3269267){$3$}
\rput[bl](1.9999998,-3.7269268){$1$}
\rput[tl](0.0,-3.3269267){$0$}
\rput[b](3.6319997,-2.1749268){$4$}
\rput[bl](1.9999998,-1.3269268){$5$}
\rput[bl](0.43199974,-2.1749268){$6$}
\rput[t](3.1999998,-4.526927){$2$}
\psline[linecolor=black, linewidth=0.04](0.49599975,-3.182927)(3.6959999,-3.182927)
\psline[linecolor=black, linewidth=0.04](0.8959997,-1.9829268)(3.2959998,-4.382927)
\psline[linecolor=black, linewidth=0.04](2.0959997,-1.5829268)(2.0959997,-3.182927)(3.2959998,-1.9829268)
\psarc[linecolor=black, linewidth=0.04, dimen=outer](2.0959997,-3.5829268){0.4}{94.2364}{86.63354}
\psarc[linecolor=black, linewidth=0.04, dimen=outer](3.4959998,-2.5829268){1.8}{262.20142}{143.17828}
\psarc[linecolor=black, linewidth=0.04, dimen=outer](2.8959997,-2.3829267){1.2}{-41.992718}{-224.12488}
\psline[linecolor=red, linewidth=0.04](0.49599975,-3.182927)(2.0959997,-3.182927)(3.2959998,-4.382927)
\psdots[linecolor=black, dotsize=0.2](3.2359998,-4.322927)
\psdots[linecolor=black, dotsize=0.2](0.49599975,-3.182927)
\psdots[linecolor=black, dotsize=0.2](2.0959997,-3.182927)
\end{pspicture}
}\vspace{0.4cm}
    \caption{A Hamiltonian path in $G_7$.}
    \end{subfigure}
    ~ 
    \begin{subfigure}[b]{0.45\textwidth}
    \centering
\psscalebox{1.0 1.0} 
{
\begin{pspicture}(0,-4.3999567)(4.84034,0.44046494)
\psarc[linecolor=black, linewidth=0.04, dimen=outer](3.2601764,-1.9397893){0.4}{47.16669}{37.07866}
\rput[bl](2.3241763,-3.3237894){$1$}
\rput[tl](0.3241764,-2.9237893){$0$}
\rput[t](4.3241763,-2.9237893){$2$}
\psline[linecolor=black, linewidth=0.04](0.82017636,-2.7797892)(4.0201764,-2.7797892)
\psline[linecolor=black, linewidth=0.04](2.4201765,-2.7797892)(2.4201765,-1.1797893)
\psline[linecolor=black, linewidth=0.04](1.2201763,-1.5797893)(2.4201765,-2.7797892)(3.6201763,-1.5797893)
\psarc[linecolor=red, linewidth=0.04, dimen=outer](2.4201765,-2.7797892){1.6}{179.40318}{134.46336}
\rput[b](3.9561763,-1.7717893){$3$}
\rput[bl](2.3241763,-0.92378926){$4$}
\rput[bl](0.7561764,-1.7717893){$5$}
\psarc[linecolor=black, linewidth=0.04, dimen=outer](3.2201765,-1.9797893){1.2}{-41.992718}{-224.12488}
\psarc[linecolor=black, linewidth=0.04, dimen=outer](3.0201764,-1.3797892){1.8}{-54.17327}{-172.55327}
\psdots[linecolor=black, dotsize=0.24](4.0201764,-2.7797892)
\psarc[linecolor=black, linewidth=0.04, dimen=outer](3.6201763,-2.7797892){0.4}{0.0}{357.70938}
\psarc[linecolor=black, linewidth=0.04, dimen=outer](2.4201765,-3.1797893){0.4}{90.0}{82.04599}
\psarc[linecolor=black, linewidth=0.04, dimen=outer](1.8201764,-1.3797892){1.8}{-6.5632415}{-125.00865}
\psarc[linecolor=black, linewidth=0.04, dimen=outer](2.4201765,-1.5797893){1.2}{-2.5447211}{-178.99937}
\psdots[linecolor=black, dotsize=0.24](3.5601764,-1.6397892)
\psdots[linecolor=black, dotsize=0.24](2.4201765,-1.1797893)
\psline[linecolor=red, linewidth=0.04](1.2201763,-1.5797893)(2.4201765,-2.7797892)(0.82017636,-2.7797892)
\psdots[linecolor=black, dotsize=0.24](1.2801764,-1.6397892)
\psdots[linecolor=black, dotsize=0.24](0.82017636,-2.7797892)
\psdots[linecolor=black, dotsize=0.24](2.4201765,-2.7797892)
\end{pspicture}
}
    \caption{A Hamiltonian cycle in $G_6$.}
    \end{subfigure}
    \caption{Traceable and Hamiltonian graphs.}
\end{figure}
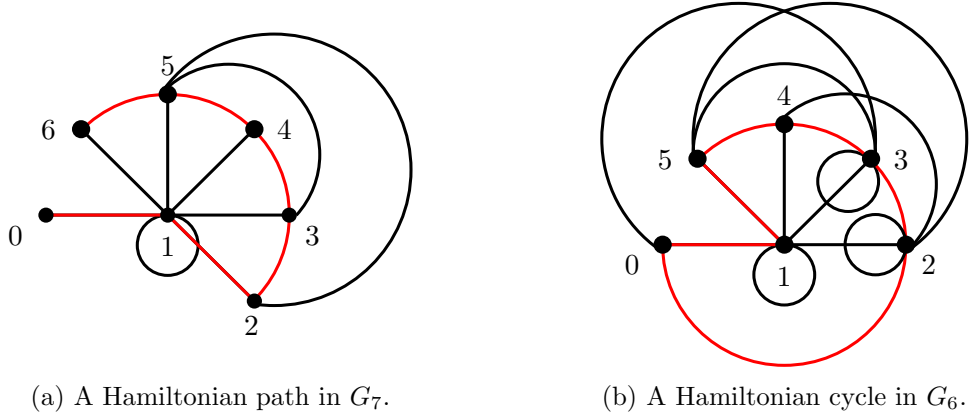

Planarity of graphs and graph coloring are also examined in this section. First, we show that the graph $G_n$ contains a complete graph of appropriate order and extend the results to describe the planarity of graphs and graph coloring.

\begin{theorem}\label{ThmContainMK2}
Let $p$ and $q$ be prime such that $p < q$, and $m$ be the number of prime numbers which are less than $pq$ and different from $p$ and $q$. Moreover, suppose that $k$ is the largest positive number such that $p^{k} < pq$. Then $G_{pq}$ contains a maximal complete subgraph of order $m+k+2$.
\end{theorem}
\begin{proof}
    Suppose that $p_1, p_2,\ldots,p_{m}, p, q$ are all distinct prime numbers less than $pq$. Then the greatest common divisor of every pair of vertices $1, p_1, p_2,\ldots,p_m, p, p^2,\ldots, p^{k}, q$ divides $pq$, so they are adjacent to each others. Thus these vertices form a complete subgraph $K_{m+k+2}$ of $G_{pq}$. Let $a\in \mathbb{Z}_{pq}$ such that $a\not\in \{1, p_1, p_2,\ldots,p_m, p, p^2,\ldots, p^{k}, q\}$. This implies that there is $p_j$ dividing $a$ for some $j= 1,2,\ldots,m$. Thus $\gcd(a,p_j) = p_j \nmid pq$, so $\{a, p_j\}\not\in \nu_{pq}$. This implies that $G_{pq}$ contains a maximal complete subgraph of order $m+k+2$.
\end{proof}

\begin{theorem}\label{ThmContainMK}
Let $k$ be a positive integer, $p$ be prime, and $m$ be the number of prime numbers which are less than $p^k$ and different from $p$. If $p= 2, k= 1$, then $G_{p^k}$ contains a maximal complete subgraph of order $k+1$. Otherwise, $G_{p^k}$ contains a maximal complete subgraph of order $m+k$.
\end{theorem}
\begin{proof}
    The case $p= 2, k= 1$ is obvious. For the rest we suppose that $p_1, p_2,\ldots,p_{m}, p$ are all distinct prime numbers less than $p^k$. Then the greatest common divisor of every pair of vertices $1, p_1, p_2,\ldots,p_m, p, p^2,\ldots, p^{k-1}$ divides $p^k$, so they are adjacent to each others. Thus these vertices form a complete subgraph $K_{m+k}$ of $G_{p^k}$. Let $a\in \mathbb{Z}_{p^k}$ such that $a\not\in \{1, p_1,\ldots, p_{m}, p, p^2,\ldots, p^{k-1}\} $. This implies that there is $p_j$ dividing $a$ for some $j= 1,2,\ldots,m$. Thus $\gcd(a,p_j) = p_j \nmid p^k$, so $\{a, p_j\}\not\in \nu_p$. This implies that $G_{p^k}$ contains a maximal complete subgraph of order $m+k$.
\end{proof}

\begin{corollary}\label{CorContainM1}
    $G_n$ contains a complete subgraph of order $m+1$ where $m$ is the number of prime numbers less than $n$.
\end{corollary}

\begin{figure}[!h]
    \centering
    \begin{subfigure}[b]{0.45\textwidth}
    \centering
\psscalebox{1.0 1.0} 
{
\begin{pspicture}(0,-4.7969275)(5.3162513,-0.7627815)
\psarc[linecolor=black, linewidth=0.04, dimen=outer](2.0959997,-3.1829271){1.6}{-41.987213}{-223.7636}
\psdots[linecolor=black, dotsize=0.24](0.95599973,-2.042927)
\psdots[linecolor=black, dotsize=0.24](3.2359998,-2.042927)
\rput[t](3.9999998,-3.3269272){$3$}
\rput[bl](1.9999998,-3.726927){$1$}
\rput[tl](0.0,-3.3269272){$0$}
\rput[b](3.6319997,-2.174927){$4$}
\rput[bl](1.9999998,-1.3269271){$5$}
\rput[bl](0.43199974,-2.174927){$6$}
\rput[t](3.1999998,-4.526927){$2$}
\psline[linecolor=black, linewidth=0.04](0.49599975,-3.1829271)(3.6959999,-3.1829271)
\psline[linecolor=black, linewidth=0.04](0.8959997,-1.9829271)(3.2959998,-4.382927)
\psline[linecolor=black, linewidth=0.04](2.0959997,-1.5829271)(2.0959997,-3.1829271)(3.2959998,-1.9829271)
\psarc[linecolor=black, linewidth=0.04, dimen=outer](2.0959997,-3.582927){0.4}{94.2364}{86.63354}
\psarc[linecolor=red, linewidth=0.04, dimen=outer](3.4959998,-2.582927){1.8}{-97.7986}{-216.82172}
\psarc[linecolor=red, linewidth=0.04, dimen=outer](2.8959997,-2.3829272){1.2}{-41.992718}{-224.12488}
\psdots[linecolor=black, dotsize=0.2](0.49599975,-3.1829271)
\psline[linecolor=red, linewidth=0.04](3.2959998,-4.382927)(2.0959997,-3.1829271)(2.0959997,-1.5829271)
\psarc[linecolor=red, linewidth=0.04, dimen=outer](2.0959997,-3.1829271){1.6}{315.0}{0.33532384}
\psline[linecolor=red, linewidth=0.04](2.0959997,-3.1829271)(3.6959999,-3.1829271)
\psdots[linecolor=black, dotsize=0.2](3.2359998,-4.322927)
\psdots[linecolor=black, dotsize=0.2](3.6959999,-3.1829271)
\psdots[linecolor=black, dotsize=0.24](2.0959997,-1.5829271)
\psdots[linecolor=black, dotsize=0.2](2.0959997,-3.1829271)
\end{pspicture}
}\vspace{0.4cm}
    \caption{A maximal complete subgraph $K_4$ of $G_7$.}
    \end{subfigure}
    ~ 
    \begin{subfigure}[b]{0.45\textwidth}
    \centering
\psscalebox{1.0 1.0} 
{
\begin{pspicture}(0,-4.3999567)(4.84034,0.44046494)
\psarc[linecolor=black, linewidth=0.04, dimen=outer](3.2601764,-1.9397893){0.4}{47.16669}{37.07866}
\rput[bl](2.3241763,-3.3237894){$1$}
\rput[tl](0.3241764,-2.9237893){$0$}
\rput[t](4.3241763,-2.9237893){$2$}
\psline[linecolor=black, linewidth=0.04](0.82017636,-2.7797892)(4.0201764,-2.7797892)
\psline[linecolor=red, linewidth=0.04](2.4201765,-2.7797892)(2.4201765,-1.1797893)
\psline[linecolor=red, linewidth=0.04](1.2201763,-1.5797893)(2.4201765,-2.7797892)(3.6201763,-1.5797893)
\psarc[linecolor=black, linewidth=0.04, dimen=outer](2.4201765,-2.7797892){1.6}{179.40318}{134.46336}
\rput[b](3.9561763,-1.7717893){$3$}
\rput[bl](2.3241763,-0.92378926){$4$}
\rput[bl](0.7561764,-1.7717893){$5$}
\psarc[linecolor=red, linewidth=0.04, dimen=outer](3.2201765,-1.9797893){1.2}{-41.992718}{-224.12488}
\psarc[linecolor=red, linewidth=0.04, dimen=outer](3.0201764,-1.3797892){1.8}{-54.17327}{-172.55327}
\psarc[linecolor=black, linewidth=0.04, dimen=outer](3.6201763,-2.7797892){0.4}{0.0}{357.70938}
\psarc[linecolor=black, linewidth=0.04, dimen=outer](2.4201765,-3.1797893){0.4}{90.0}{82.04599}
\psarc[linecolor=black, linewidth=0.04, dimen=outer](1.8201764,-1.3797892){1.8}{-6.5632415}{-125.00865}
\psarc[linecolor=red, linewidth=0.04, dimen=outer](2.4201765,-1.5797893){1.2}{-2.5447211}{-178.99937}
\psdots[linecolor=black, dotsize=0.24](0.82017636,-2.7797892)
\psline[linecolor=red, linewidth=0.04](1.2201763,-1.5797893)(2.4201765,-2.7797892)(4.0201764,-2.7797892)
\psarc[linecolor=red, linewidth=0.04, dimen=outer](2.4201765,-2.7797892){1.6}{0.0}{136.04163}
\psdots[linecolor=black, dotsize=0.24](4.0201764,-2.7797892)
\psdots[linecolor=black, dotsize=0.24](2.4201765,-2.7797892)
\psdots[linecolor=black, dotsize=0.24](3.5601764,-1.6397892)
\psdots[linecolor=black, dotsize=0.24](2.4201765,-1.1797893)
\psdots[linecolor=black, dotsize=0.24](1.2801764,-1.6397892)
\end{pspicture}
}
    \caption{A maximal complete subgraph $K_5$ of $G_6$.}
    \end{subfigure}
    \caption{Maximal complete subgraphs of $G_n$.}
\end{figure}
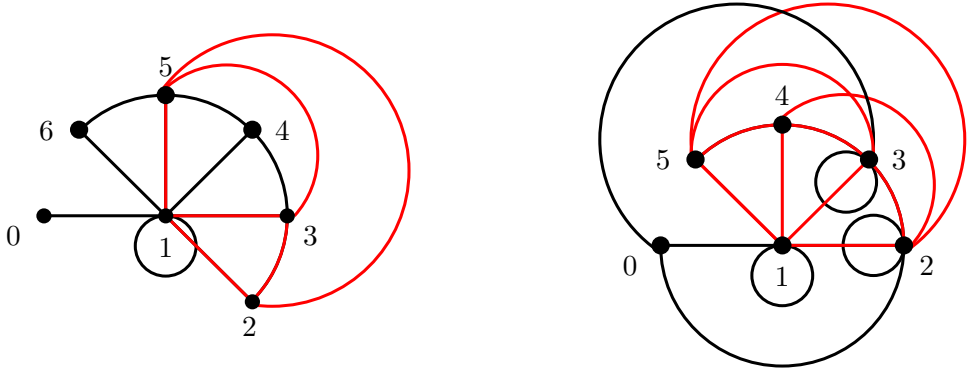

\begin{theorem}\label{ThmContainK5}
The graph $G_n$ contains a complete graph $K_5$ precisely for $n \geq 6$ and $n\neq 7$.
\end{theorem}
\begin{proof}
    We can see that the graph $G_n$ does not have a complete subgraph $K_5$ for every $n \leq 7$ and $n\neq 6$ by investigating the set $\nu_n$. For $n=6$, we see that the vertices $1, 2, 3, 4$ and $5$ are adjacent to each others forming a complete subgraph $K_5$ of $G_6$. The same technique is applied to the case $n\geq 8$ by considering the vertices $1, 2, 3, 5$ and $7$ instead. Consequently, the graph $G_n$ contains a complete graph $K_5$ precisely for $n \geq 6$ and $n\neq 7$.
\end{proof}

A graph $G$ is said to be \emph{planar} if $G$ can be drawn on the plane such that there are no any edges intersect each others only at their endpoints. Since $K_5$ is not planar, Theorem \ref{ThmContainK5} implies that the graph $G_n$ is not planar for every $n \geq 6$ and $n\neq 7$. For the case $n \leq 7$ and $n\neq 6$, we can see that the graph $G_n$ is planar by considering the set $\nu_n$. This result of planarity is concluded in the following corollary.

\begin{corollary}
The graph $G_n$ is planar precisely for $n \leq 7$ and $n\neq 6$.    
\end{corollary}

Next, we investigate some coloring results of $G_n$. The chromatic number of a graph $G$ is the minimum numbers of colors which are assigned to all vertices of $G$ in such a way that all adjacent vertices cannot be painted by the same color. To avoid the conflict with the definition of the chromatic number of a graph, all loops in $G_n$ are not considered in this situation.

\begin{corollary}
Let $G_n = (\mathbb{Z}_n, \nu_n)$ be a graph. Then;
\begin{enumerate}
    \item $\chi(G_2)= \chi(G_3) = 2, \chi(G_4)= \chi(G_5) = 3, \chi(G_6) = 5, \chi(G_7) = 4$.
    \item $\chi(G_{pq}) \geq m+k+2$ where the conditions of $p, q, m$ and $k$ are imitated from Theorem \ref{ThmContainMK2}.
    \item $\chi(G_{p^k}) \geq m+k$ where the conditions of $p, m$ and $k$ are imitated from Theorem \ref{ThmContainMK}.
    \item $\chi(G_n) \geq m+1$ where $m$ is the number of all prime numbers less than $n$.
\end{enumerate}
\end{corollary}
\begin{proof}
    The result of the case $n\leq 7$ can be obtained by consider the set $\nu_n$. The other three items are obtained directly by Theorem \ref{ThmContainMK2}, Theorem \ref{ThmContainMK}, and Corollary \ref{CorContainM1}, respectively.
\end{proof}

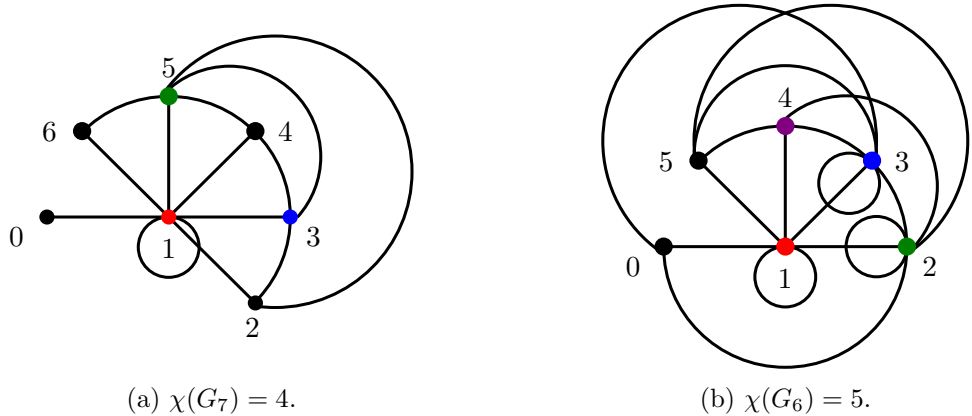
\begin{figure}[!h]
    \centering
    \begin{subfigure}[b]{0.45\textwidth}
    \centering
\psscalebox{1.0 1.0} 
{
\begin{pspicture}(0,-4.7969275)(5.3162513,-0.7627815)
\definecolor{colour1}{rgb}{0.0,0.5019608,0.0}
\psarc[linecolor=black, linewidth=0.04, dimen=outer](2.0959997,-3.1829271){1.6}{-41.987213}{-223.7636}
\psdots[linecolor=black, dotsize=0.24](0.95599973,-2.042927)
\psdots[linecolor=black, dotsize=0.24](3.2359998,-2.042927)
\rput[t](3.9999998,-3.3269272){$3$}
\rput[bl](1.9999998,-3.726927){$1$}
\rput[tl](0.0,-3.3269272){$0$}
\rput[b](3.6319997,-2.174927){$4$}
\rput[bl](1.9999998,-1.3269271){$5$}
\rput[bl](0.43199974,-2.174927){$6$}
\rput[t](3.1999998,-4.526927){$2$}
\psline[linecolor=black, linewidth=0.04](0.49599975,-3.1829271)(3.6959999,-3.1829271)
\psline[linecolor=black, linewidth=0.04](0.8959997,-1.9829271)(3.2959998,-4.382927)
\psline[linecolor=black, linewidth=0.04](2.0959997,-1.5829271)(2.0959997,-3.1829271)(3.2959998,-1.9829271)
\psarc[linecolor=black, linewidth=0.04, dimen=outer](2.0959997,-3.582927){0.4}{94.2364}{86.63354}
\psarc[linecolor=black, linewidth=0.04, dimen=outer](3.4959998,-2.582927){1.8}{-97.7986}{-216.82172}
\psarc[linecolor=black, linewidth=0.04, dimen=outer](2.8959997,-2.3829272){1.2}{-41.992718}{-224.12488}
\psdots[linecolor=black, dotsize=0.2](0.49599975,-3.1829271)
\psdots[linecolor=black, dotsize=0.2](3.2359998,-4.322927)
\psdots[linecolor=blue, dotsize=0.2](3.6959999,-3.1829271)
\psdots[linecolor=colour1, dotsize=0.24](2.0959997,-1.5829271)
\psdots[linecolor=red, dotsize=0.2](2.0959997,-3.1829271)
\end{pspicture}
}\vspace{0.4cm}
    \caption{$\chi(G_7)= 4$.}
    \end{subfigure}
    ~ 
    \begin{subfigure}[b]{0.45\textwidth}
    \centering
\psscalebox{1.0 1.0} 
{
\begin{pspicture}(0,-4.3999567)(4.84034,0.44046494)
\definecolor{colour0}{rgb}{0.5019608,0.0,0.5019608}
\definecolor{colour1}{rgb}{0.0,0.5019608,0.0}
\psarc[linecolor=black, linewidth=0.04, dimen=outer](3.2601764,-1.9397893){0.4}{47.16669}{37.07866}
\rput[bl](2.3241763,-3.3237894){$1$}
\rput[tl](0.3241764,-2.9237893){$0$}
\rput[t](4.3241763,-2.9237893){$2$}
\psline[linecolor=black, linewidth=0.04](0.82017636,-2.7797892)(4.0201764,-2.7797892)
\psline[linecolor=black, linewidth=0.04](2.4201765,-2.7797892)(2.4201765,-1.1797893)
\psline[linecolor=black, linewidth=0.04](1.2201763,-1.5797893)(2.4201765,-2.7797892)(3.6201763,-1.5797893)
\psarc[linecolor=black, linewidth=0.04, dimen=outer](2.4201765,-2.7797892){1.6}{179.40318}{134.46336}
\rput[b](3.9561763,-1.7717893){$3$}
\rput[bl](2.3241763,-0.92378926){$4$}
\rput[bl](0.7561764,-1.7717893){$5$}
\psarc[linecolor=black, linewidth=0.04, dimen=outer](3.2201765,-1.9797893){1.2}{-41.992718}{-224.12488}
\psarc[linecolor=black, linewidth=0.04, dimen=outer](3.0201764,-1.3797892){1.8}{-54.17327}{-172.55327}
\psarc[linecolor=black, linewidth=0.04, dimen=outer](3.6201763,-2.7797892){0.4}{0.0}{357.70938}
\psarc[linecolor=black, linewidth=0.04, dimen=outer](2.4201765,-3.1797893){0.4}{90.0}{82.04599}
\psarc[linecolor=black, linewidth=0.04, dimen=outer](1.8201764,-1.3797892){1.8}{-6.5632415}{-125.00865}
\psarc[linecolor=black, linewidth=0.04, dimen=outer](2.4201765,-1.5797893){1.2}{-2.5447211}{-178.99937}
\psdots[linecolor=blue, dotsize=0.24](3.5601764,-1.6397892)
\psdots[linecolor=colour0, dotsize=0.24](2.4201765,-1.1797893)
\psdots[linecolor=black, dotsize=0.24](1.2801764,-1.6397892)
\psdots[linecolor=black, dotsize=0.24](0.82017636,-2.7797892)
\psdots[linecolor=red, dotsize=0.24](2.4201765,-2.7797892)
\psdots[linecolor=colour1, dotsize=0.24](4.0201764,-2.7797892)
\end{pspicture}
}
    \caption{$\chi(G_6)= 5$.}
    \end{subfigure}
    \caption{The chromatic number of $G_n$.}
\end{figure}

\bibliographystyle{ieeetr}
\bibliography{reference}



\end{document}